\documentclass[a4paper,11pt,twoside]{amsart}
\usepackage[english]{babel}
\usepackage[utf8]{inputenc}

\usepackage[a4paper,inner=3cm,outer=3cm,top=4cm,bottom=4cm,pdftex]{geometry}
\usepackage{fancyhdr}
\usepackage{color}
\usepackage{bold-extra}
\usepackage{ mathrsfs }

\usepackage{comment}
\usepackage{graphics}
\usepackage{aliascnt}
\usepackage[pdftex,citecolor=green,linkcolor=red]{hyperref}

\usepackage{amsmath}
\usepackage{amsfonts}
\usepackage{amssymb}
\usepackage{amsthm}
\usepackage{comment}
\usepackage{mathtools}
\usepackage{enumerate}
\usepackage{enumitem} 

\newtheorem{theorem}{Theorem}[section]

\newtheorem{lemma}[theorem]{Lemma}
\newtheorem{proposition}[theorem]{Proposition}
\theoremstyle{definition}

\numberwithin{equation}{section}
\newtheorem*{theorem*}{Theorem}

\newcommand\eps{\varepsilon}

\newcommand\F{\mathbb{F}}

\newcommand\R{\mathbb{R}}
\newcommand\Z{\mathbb{Z}}

\begin{document}
\title{Burgess-type character sum estimates over generalized arithmetic progressions of rank $2$}

\author[Alsetri]{Ali Alsetri}
\address{Department of Mathematics, University of Kentucky\\
715 Patterson Office Tower\\
Lexington, KY 40506\\
USA}
\email{alialsetri@uky.edu}

\author[Shao]{Xuancheng Shao}
\address{Department of Mathematics, University of Kentucky\\
715 Patterson Office Tower\\
Lexington, KY 40506\\
USA}
\email{xuancheng.shao@uky.edu}

\thanks{XS was supported by NSF grant DMS-2452462. Thanks to the anonymous referee for helpful comments and suggestions.}



\maketitle
\begin{abstract}
We extend the classical Burgess estimates to character sums over proper generalized arithmetic progressions (GAPs) of rank $2$ in prime fields $\mathbb{F}_p$. The core of our proof is a sharp upper bound for the multiplicative energy of these sets, established by adapting an argument of Konyagin and leveraging tools from the geometry of numbers. A key step in our argument involves establishing new upper bounds for the sizes of Bohr sets, which may be of independent interest.
\end{abstract}

\section{Introduction}

Let $p$ be prime and let $\chi\pmod{p}$ be a nontrivial Dirichlet character. In this paper we are concerned with character sums of the form $\sum_{n \in A} \chi(n)$, where $A \subset \F_p$ is a subset with additive structures. In the classical case when $A$ is an interval or an arithmetic progression, the behavior of these character sums  is central to understanding the distribution of primes, quadratic residues, primitive roots, etc.

In $1918$, P\'{o}lya and Vinogradov \cite{Polya,Vinogradov} independently proved the following general bound for character sums over intervals.

\begin{theorem*}[P\'{o}lya-Vinogadov] 
Let $p$ be prime and let $\chi\pmod{p}$ be a nontrivial Dirichlet character. Let $I \subset \F_p$ be an interval. Then
$$
\Big|\sum_{n \in I}\chi(n)\Big| \ll p^{1/2}\log p.
$$
\end{theorem*}

The P\'{o}lya-Vinogradov bound is sharp up to the $\log p$ factor and represents square-root cancellation in the character sum. The factor $\log p$ can be improved if $\chi$ has odd order or if GRH is assumed; see \cite{montgomery-vaughan,granville-sound,goldmakher,lamzouri-mangerel}.

However, the P\'{o}lya-Vinogradov estimate becomes trivial when $|I| \ll p^{1/2}\log p$. In a series of papers starting in 1962, Burgess \cite{burgess1} broke through the P\'{o}lya-Vinogradov barrier for short intervals.

\begin{theorem*}[Burgess] 
Let $p$ be prime and let $\chi\pmod{p}$ be a nontrivial Dirichlet character. Let $I \subset \F_p$ be an interval. Then for any positive integer $r \geq 2$ and any $\eps > 0$ we have
$$
\Big|\sum_{n \in I}\chi(n)\Big| \ll_{\eps,r} |I|^{1-1/r}p^{(r+1)/(4r^2)+\eps}.
$$
In particular, if $|I| \geq p^{1/4+\eps}$ for any $\eps > 0$ then
$$
\Big|\sum_{n \in I}\chi(n)\Big| \ll_{\eps} p^{-\delta}|I|
$$
for some positive constant $\delta = \delta(\eps) > 0$.
\end{theorem*}

Assuming GRH, nontrivial estimates for the character sum can be obtained as long as $|I| \geq p^{\eps}$ for any $\eps > 0$. Nevertheless, the exponent $1/4$ in the Burgess estimate remains the state of the art to this day. Despite the fact that the P\'{o}lya-Vinogradov inequality and the Burgess estimate treat character sums over intervals of lengths in different regimes, there are still strong connections between them; see \cite{fromm-goldmakher,mangerel,granville-mangerel}.

\subsection{Character sum estimates over GAPs}

In this paper, we focus on Burgess-type estimates for character sums over generalized arithmetic progressions. A generalized arithmetic progression (GAP) of rank $d$ in $\F_p$ is a set $A \subset \F_p$ of the form 
$$
A = \{a_0 + a_1x_1 + \dots + a_dx_d : 1 \leq x_i \leq H_i\}
$$
for positive integers $H_1,\dots,H_d$ and elements $a_0,\dots,a_d \in \F_p$ with $a_1,\dots,a_d \neq 0$. The GAP $A$ is said to be proper if $|A| = \prod_{i=1}^d H_i$.  In additive combinatorics, GAPs serve as primary examples of highly structured sets and naturally arise when studying sets with small sumsets as codified by Freiman's theorem; see \cite{TaoVu}.

Clearly, GAPs of rank $1$ are precisely intervals and arithmetic progressions modulo $p$.
Our main result of this paper is a generalization of the Burgess estimate to GAPs of rank $2$.

\begin{theorem}\label{char-sum}
Let $p$ be prime and let $\chi\pmod{p}$ be a nontrivial Dirichlet character. Let $A \subset \F_p$ be a proper GAP of rank $2$. If $|A| \geq p^{1/4+\eps}$ for any $\eps > 0$ then
$$
\Big|\sum_{n \in A}\chi(n)\Big| \ll_{\eps} p^{-\delta}|A|
$$
for some positive constant $\delta = \delta(\eps) > 0$.
\end{theorem}

To put our results into perspective, Chang \cite{Chang} obtained nontrivial bounds for character sums over GAPs of any fixed rank $d$, provided that $|A| \geq p^{2/5+\eps}$. See also \cite{Hanson} for analogous results when $A$ is a Bohr set. The exponent $2/5$ in Chang's result was improved to $1/3$ in \cite{Volostnov, Schoen-Shkredov}, building on an earlier improvement in \cite{Volostnov-Shkredov}. In the special case of $d=2$, the exponent of $1/3$ also follows from \cite[Corollary 2.5]{alsetri-shao} which relies on results in \cite{shao,Heath-Brown}.

 An alternative line of work has studied extensions of Burgess' method to character sums over short boxes in finite fields $\F_{p^d}$. Given a basis $\{\omega_1,\omega_2,\dots,\omega_d\}$ for the $d$-dimensional vector space $\F_{p^d}$ over $\F_p$, consider boxes $B \subset \F_{p^d}$ of the form
 $$
 B = \{\omega_1x_1 + \cdots + \omega_dx_d : x_i \in I_i\},
 $$
 where each $I_i \subset \F_p$ is an interval. For Burgess-type character sum estimates over such boxes, see \cite{burgess2, karacuba1, karacuba2, DL, Chang, Chang2, konyagin, Gabdullin}. In particular, Konyagin \cite{konyagin} obtained nontrivial estimates for the character sum when $|I_i| \geq p^{1/4+\eps}$ for each $i$.

\subsection{Multiplicative energies of GAPs in $\F_p$}
  
In obtaining Burgess-type character sum estimates over a GAP $A \subset \F_p$, a crucial role is played by the multiplicative energy $E_{\times}(A)$ of $A$ defined as
$$
E_\times(A) := |\{(a_1,a_2,a_3,a_4) \in A^4 : a_1a_2 = a_3a_4 \}|.
$$
The task of estimating multiplicative energies of GAPs belongs to the fundamental concept of the sum-product phenomenon in arithmetic combinatorics, which explores the interplay between additive and multiplicative structures. In the finite field setting, it asserts that if $A \subset \F_p$ is neither too small nor too large, then either the sumset $A+A := \{a_1+a_2: a_1,a_2 \in A\}$ or the product set $A\cdot A := \{a_1a_2: a_1, a_2 \in A\}$ must be large:
$$
|A+A| + |A\cdot A| \gg |A|^{1+c},
$$
where $c>0$ is an absolute constant. Since the seminal work of Bourgain-Katz-Tao \cite{BKT}, the constant $c$ has been improved to $c=1/4$ in \cite{MS} (under mild assumptions on $|A|$). See also \cite{MPRRS, RRS, RSS} and the references therein for related sum-product type results in finite fields. 

If $A \subset \F_p$ is a GAP of rank $d$, we expect heuristically that $|A \cdot A| \gg_d |A|^{2-o(1)}$ and $E_{\times}(A) \ll_d |A|^{2+o(1)}$ provided that $|A| \leq p^{1/2}$. In this direction, the current best estimate for $E_{\times}(A)$, given by \cite[Theorem 35]{MPRRS}, is
$$
E_{\times}(A) \ll_d |A|^{32/13}
$$
provided that $|A| \leq p^{13/23}$.

The key ingredient in Theorem \ref{char-sum} is the optimal bound for multiplicative energies of GAPs in $\F_p$ of rank $2$.

\begin{theorem}\label{x-energy}
Let $p$ be prime and let $A \subset \F_p$ be a GAP of rank $2$. Then
$$ 
E_{\times}(A) \ll \Big(|A|^2 + \frac{|A|^4}{p}\Big) \log p.
$$
\end{theorem}

We conjecture that Theorem \ref{x-energy} holds for GAPs of any fixed rank $d$, which would imply Theorem \ref{char-sum} for GAPs of any fixed rank. See Section \ref{sec:general-d} for an explanation on why we are unable to prove this general case.

In comparison, Kerr \cite[Corollary 4]{Kerr} established Theorem \ref{x-energy} for rank-$d$ GAPs $A \subset \F_p$ of the form
$$
A = \{a_0+a_1x_1+\cdots + a_dx_d: 1 \leq x_i \leq H\}
$$
under the additional assumption that the dilated GAP
$$
A' = \{a_1x_1 + \cdots + a_dx_d : |x_i| \leq H^2\}
$$
is proper.

\subsection{Outline of the paper}

In Section \ref{sec:background} we record some basic results from the geometry of numbers and some classical character sum bounds. The proof of Theorem \ref{x-energy} is given in Section \ref{sec:x-energy}, which uses the geometry of numbers and follows the strategy set out by Konyagin \cite{konyagin} (also used in \cite{Gabdullin, Kerr}). In carrying out this strategy, we establish upper bounds for sizes of Bohr sets in Section \ref{sec:bohr}, which may be of independent interest. In Section \ref{sec:general-d}, we briefly explain why we are unable to generalize the argument to GAPs of rank $3$ or higher. Finally in Section \ref{sec:char-sum}, we deduce Theorem \ref{char-sum} from Theorem \ref{x-energy}.

\section{Background Results}\label{sec:background}

We start with notions and results from the geometry of numbers. Recall that if $L\subset \R^d$ is a lattice and $D \subset \R^d$ is a symmetric convex body, then for $1 \leq i \leq d$, the $i$th successive minimum $\lambda_i = \lambda_i(D, L)$ is defined to be the smallest real number $\lambda$ such that $\lambda D := \{\lambda x: x \in D\}$ contains $i$ linearly independent vectors from $L$. Clearly $\lambda_1 \leq \cdots \leq \lambda_d$. Minkowski's second theorem relates the sizes of the successive minima with $\operatorname{Vol}(D)$, the volume of $D$, and $\operatorname{Vol}(\R^d/L)$, the volume of a fundamental cell of $L$. See \cite[Theorem 3.30]{TaoVu} for a proof.

\begin{theorem}[Minkowksi's Second Theorem]\label{mink2nd}
Let $L \subset \R^d$ be a lattice, let $D \subset \R^d$ be a symmetric convex body, and let $\lambda_1,\dots,\lambda_d$ be the successive minima of $L$ with respect to $D$. Then
    $$
    \frac{\operatorname{Vol}(\R^d/L)}{\operatorname{Vol}(D)} \ll_d \lambda_1 \dots \lambda_d \ll_d \frac{\operatorname{Vol}(\R^d/L)}{\operatorname{Vol}(D)}.
    $$
\end{theorem}

The following lemma estimates the number of lattice points in $L \cap D$ in terms of the successive minima; see \cite[Exercise 3.5.6]{TaoVu}.

\begin{lemma}\label{henkcount}
Let $L \subset \R^d$ be a lattice, let $D \subset \R^d$ be a symmetric convex body, and let $\lambda_1,\dots,\lambda_d$ be the successive minima of $L$ with respect to $D$. Then
$$
\prod_{i=1}^d\max(1,\frac{1}{\lambda_i}) \ll_d |L \cap D| \ll_d \prod_{i=1}^d\max(1,\frac{1}{\lambda_i}).
$$
\end{lemma}

The polar lattice $L^*$ of a lattice $L\subset \R^d$ and the polar body $D^*$ of a symmetric convex body $D\subset \R^d$ are defined as
$$
 L^* = \{x \in \R^d : \langle x,y \rangle \in \Z \text{ for all } y \in L \},  \ \  D^* = \{x \in \R^d : \langle x, y \rangle \leq 1 \text{ for all } y \in D \},
$$
where $\langle x,y\rangle$ denotes the standard inner product on $\R^d$. The successive minima of $L$ with respect to $D$ and the successive minima of $L^*$ with respect to $D^*$ can be related by a result of Mahler \cite{Mahler}.

\begin{lemma}\label{polar-minima}
    Let $L \subset \R^d$ be a lattice, let $D \subset \R^d$ be a symmetric convex body, and let $L^*$ and $D^*$ be the polar lattice of $L$ and the polar body of $D$, respectively. Let $\lambda_1,\dots, \lambda_d$ be the successive minima of $L$ with respect to $D$ and let $\lambda_1^*,\dots, \lambda_d^*$ be the successive minima of $L^*$ with respect to $D^*$. For each $1 \leq i \leq d$ we have
    $$
    1 \ll_d \lambda_i\lambda_{d-i+1}^* \ll_d 1.
    $$
    
    \end{lemma}
    
We now turn to character sum estimates. The following lemma is a general version of Burgess' argument, which allows us to bound character sums in terms of the multiplicative energy; see \cite[Lemma 5.1]{Hanson}.

\begin{lemma}\label{char-sum-energy}
Let $p$ be prime and let $\chi \pmod{p}$ be a nontrivial Dirichlet character. Let $A,B,J \subset \F_p$ be subsets. Define
    $$
    \nu(u) = |\{(x,y) \in A \times B : xy^{-1} = u \}|
    $$
    for each $u \in \F_p$. Then for any positive integer $r$ we have
    \begin{align*}
        \sum_{u \in \F_p}\nu(u)\Big|\sum_{t \in J}\chi(u + t)\Big| \leq (|A||B|)^{1-\frac{1}{r}}(E_{\times}(A)E_\times (B))^{\frac{1}{4r}}(|J|^{2r}2r\sqrt{p} + (2r|J|)^rp )^{\frac{1}{2r}}. 
    \end{align*}
    \end{lemma}

This is our key tool in deducing Theorem \ref{char-sum} from Theorem \ref{x-energy}. It differs slightly from \cite[Lemma 5.1]{Hanson} where $\nu(u)$ is defined as $\nu(u)=\{(x,y) \in A \times B: xy = u\}$ instead. Our version follows by first removing $0$ from $B$ (if necessary) and then applying \cite[Lemma 5.1]{Hanson} with $B$ replaced by $B^{-1}$ and noting that $E_{\times}(B) = E_{\times}(B^{-1})$ when $0 \notin B$.

\section{Bounding the sizes of Bohr sets}\label{sec:bohr}

Let $p$ be prime. Let $a_1,\cdots,a_d \in \F_p\setminus\{0\}$ and let $\eta_1,\cdots,\eta_d \in (0,1/2)$. Let $\Gamma = (a_1,\cdots,a_d)$, $\eta = (\eta_1,\cdots,\eta_d)$, and define the Bohr set
$$
B = B(\Gamma,\eta) := \{x \in \F_p: \|a_ix/p\| \leq \eta_i\text{ for each }1 \leq i \leq d\},
$$
where $\|\cdot\|$ denotes the distance to the nearest integer. By the pigeonhole principle, one has the lower bound
$$
|B| \gg_d (\eta_1\cdots \eta_d) p.
$$
(See \cite[Lemma 4.20]{TaoVu} for a proof in the case $\eta_1=\cdots=\eta_d$). In the other direction, we have the trivial upper bound
$$
|B| \ll \min(\eta_1,\cdots,\eta_d)p,
$$
which follows from simply considering one of the conditions $\|a_ix/p\| \leq \eta_i$. This upper bound is sharp when $a_1 = \cdots = a_d$.

We are interested in obtaining non-trivial upper bounds for $|B|$ when $a_1,\cdots,a_d$ are assumed to satisfy certain non-degeneracy conditions. To describe our bounds, we need to define two quantities $t(\Gamma,\eta)$ and $\delta(\Gamma,\eta)$ as follows. Define the box $R = R_{\eta} \subset \R^d$ by
$$
R_{\eta} = [-\eta_1,\eta_1] \times \cdots \times [-\eta_d,\eta_d]
$$ 
and the lattice $L = L_{\Gamma} \subset \R^d$ by
$$
L_{\Gamma} = \Z^d + \{ p^{-1}(a_1x, \cdots, a_dx): x \in \Z\}.
$$
Note that there is a one-to-one correspondence between $B(\Gamma,\eta)$ and $R_{\eta} \cap L_{\Gamma}$ by sending $x \in B(\Gamma,\eta)$ to the integral shift of $(a_1x/p, \cdots,a_dx/p)$ which lies in $R_\eta$. 

Define $t(\Gamma,\eta)$ to be the largest positive integer $t$ such that there are $t$ linearly independent vectors in $R_{\eta} \cap L_{\Gamma}$. Clearly $0 \leq t \leq d$, and $t > 0$ if $B$ contains a non-zero element. Intuitively one can think of $t(\Gamma,\eta)$ as the ``true dimension" of the Bohr set $B$.

Define $\delta(\Gamma,\eta)$ to be the supremum of all real numbers $\delta$ such that the equation
$$
a_1u_1 + \cdots + a_du_d \equiv 0\pmod{p}, \ \ u_i \in \Z\text{ and }|u_i| \leq\delta/\eta_i\text{ for each }1 \leq i \leq d
$$
only has the trivial solution $u_1=\cdots=u_d=0$. Clearly $\min \eta_i \leq \delta \leq p\cdot \max\eta_i$. We will see in the proof of Proposition \ref{prop:bohr} that if $t < d$ then $\delta \ll_d 1$. Note that if $\eta_1=\cdots=\eta_d$ and $\delta > \eta_1$, then all $a_1,\cdots,a_d$ must be distinct, ruling out the most degenerate situation. Intuitively, the larger the quantity $\delta(\Gamma,\eta)$, the more ``independent" the frequencies $a_1,\cdots,a_d$ are.

\begin{proposition}\label{prop:bohr}
Let $p$ be prime. Let $a_1,\cdots,a_d \in \F_p\setminus\{0\}$ and let $\eta_1,\cdots,\eta_d \in (0,1/2)$. Let $\Gamma = (a_1,\cdots,a_d)$ and $\eta = (\eta_1,\cdots,\eta_d)$. Define the Bohr set $B = B(\Gamma,\eta)$ and the quantities $t = t(\Gamma,\eta)$, $\delta = \delta(\Gamma,\eta)$ as above. Then
$$
|B| \ll_d \delta^{t-d} (\eta_1\cdots\eta_d)p.
$$
\end{proposition}

In comparison, earlier results \cite[Lemma 13]{Kerr} or \cite[Proposition 2.1]{Shao-PV} give
$$
|B| \ll \max(1, \delta^{-d}) (\eta_1\cdots \eta_d)p.
$$
When $t < d$, our Proposition \ref{prop:bohr} saves an extra factor of $\delta^t$ compared to the previous bound. It is essentially this saving which allows us to remove the additional properness assumption on dilates of $A$ in \cite[Theorem 3]{Kerr} when $d=2$.

Note that if $t=d$ then our upper bound for $|B|$ in Proposition \ref{prop:bohr} matches the lower bound (up to constants). If $t < d$, we expect that the factor $\delta^{t-d}$ in the upper bound to be sharp (up to constants). In the case $t=1$ this is demonstrated by the following example. Take $\eta_1=\cdots = \eta_d = \eta$ and $a_i \asymp (\delta/\eta)^{i-1}$ for $1 \leq i \leq d$. Then $B$ contains the interval $[1, p\eta/a_d]$ and hence
$$
|B| \gg p\eta (\delta/\eta)^{-(d-1)} = \delta^{1-d} \eta^d p.
$$

\begin{proof}[Proof of Proposition \ref{prop:bohr}]
In this proof, we allow all implied constants to depend on $d$.
As defined earlier, consider the box $R = R_{\eta} \subset \R^d$ defined by
$$
R = [-\eta_1,\eta_1] \times \cdots \times [-\eta_d,\eta_d],
$$ 
and the lattice $L = L_{\Gamma} \subset \R^d$ defined by
$$
L = \Z^d + \{ p^{-1}(a_1x, \cdots, a_dx): x \in \Z\}.
$$
For $1 \leq i \leq d$, let $\mu_i$ be the $i$th successive minimum of $R$ with respect to $L$. By Minkowski's second theorem (Theorem \ref{mink2nd}), we have
$$
\mu_1\mu_2\cdots\mu_d \asymp \frac{1}{(\eta_1\cdots \eta_d)p}.
$$
By Lemma \ref{henkcount}, we have
$$
|B| = |R \cap L| \ll \prod_{j=1}^d \max(1, \mu_j^{-1}).
$$
Consider also the dual body
$$
R^* = R_{\eta}^* = \{(u_1,\cdots,u_d) \in \R^d: \eta_1|u_1|+\cdots+\eta_d|u_d| \leq 1\}
$$
and the dual lattice
$$
L^* = L_{\Gamma}^* = \{(x_1,\cdots,x_d) \in \Z^d: a_1x_1+\cdots+a_dx_d \equiv 0\pmod{p}\}.
$$
By the definition of $\delta$, $\eps R^* \cap L^* = \{0\}$ for any $\eps < \delta$. Hence the first successive minimum of $R^*$ with respect to $L^*$ satisfies $\mu_1^* \geq \delta$. Hence $\mu_{d} \ll \delta^{-1}$ by Lemma \ref{polar-minima}. By the definition of $t$, we have $\mu_t \leq 1$ and $\mu_{t+1}>1$. It follows that
$$
|B| \ll (\mu_1\cdots\mu_t)^{-1} = (\mu_1\cdots\mu_d)^{-1} (\mu_{t+1}\cdots\mu_d) \ll (\eta_1\cdots\eta_d) p \cdot (\delta^{-1})^{d-t}.
$$
This completes the proof.
\end{proof}

\section{Proof of Theorem \ref{x-energy}}\label{sec:x-energy}

In this section, we prove Theorem \ref{x-energy} which establishes sharp upper bounds on the multiplicative energy of a GAP $A$ of rank $2$ in $\F_p$. First we reduce to the case when $A$ is symmetric and proper.

\begin{proposition}\label{prop:x-energy-symmetric}
Let $p$ be prime and let $A \subset \F_p$ be a symmetric proper GAP of rank $2$. Then
$$ 
E_{\times}(A) \ll \Big(|A|^2 + \frac{|A|^4}{p}\Big) \log p.
$$
\end{proposition}

\begin{proof}[Proof of Theorem \ref{x-energy} assuming Proposition \ref{prop:x-energy-symmetric}]
Let $A \subset \F_p$ be a GAP of rank $2$. We can find a proper GAP $B \subset \F_p$ of rank at most $2$ containing $A$ with $B-B$ also proper and $|B| \ll |A|$. This follows, for example, by applying \cite[Corollary 1.18]{TaoVu-GAP} to a suitable translate of $A$. It suffices to show that
$$
E_{\times}(B) \ll \Big(|B|^2 + \frac{|B|^4}{p}\Big)\log p.
$$

For each $z \in \F_p$, let $r(z)$ be the number of solutions to $yz=x$ with $x,y \in B$ and let $r'(z)$ be the number of solutions to $y'z=x'$ for some $x',y' \in B-B$. We claim that if $r(z) > 0$ then $r(z) \leq r'(z)$.

To see this, suppose that $r(z) > 0$. Choose $x_0,y_0\in B$ with $y_0z=x_0$. For each representation $yz=x$ with $x,y \in B$, we obtain a representation $(y-y_0)z=x-x_0$, where $y-y_0, x-x_0 \in B-B$. Hence $r(z) \leq r'(z)$, as claimed.

Since there are $O(|B|^2)$ solutions to $xz=yw$ with $x,y,z,w \in B$ and at least one of $x,y,z,w$ being zero, we have
$$
E_{\times}(B) = \sum_{z \in \F_p}r(z)^2 + O(|B|^2) \leq \sum_{z \in \F_p}r'(z)^2 + O(|B|^2) \leq E_{\times}(B-B) + O(|B|^2).
$$
Since $B-B$ is proper by construction and $|B-B| \ll |B|$, we may apply Proposition \ref{prop:x-energy-symmetric} to get
$$
E_{\times}(B-B) \ll \Big(|B|^2 + \frac{|B|^4}{p}\Big)\log p.
$$
This leads to the desired bound for $E_{\times}(B)$.
\end{proof}

In the remainder of this section we prove Proposition \ref{prop:x-energy-symmetric}. Let $A$ be a symmetric proper GAP of rank $2$ of the form
$$
A = \{a_1x_1 + a_2x_2: |x_i| \leq H_i\}
$$
for positive integers $H_1,H_2$ and elements $a_1,a_2 \in \F_p\setminus\{0\}$. Without loss of generality, assume that $H_1 \leq H_2$. The properness of $A$ implies that $|A| \asymp H_1H_2$. For each $z \in \F_p$, let $r(z)$ be the number of solutions to $yz=x$ with $x,y \in A$. It suffices to prove that
$$ \sum_{z \in \F_p} r(z)^2 \ll \Big(|A|^{2} + \frac{|A|^4}{p} \Big)\log p.
$$
We can interpret $r(z)$ as the number of lattice points in a convex body as follows. Define the lattice $\Gamma_z \subset \Z^{4}$ by
$$ 
\Gamma_z = \{(x_1,x_2,y_1,y_2) \in \Z^{4}: z(a_1x_1+a_2x_2) \equiv a_1y_1+a_2y_2\pmod{p}\}
$$
and define the box $D \subset \R^{4}$ by
$$
D = \{(x_1,x_2,y_1,y_2) \in \R^{4}: |x_1|, |y_1| \leq H_1 \textnormal{ and } |x_2|,|y_2| \leq H_2\}.
$$
Then we have
$$ 
r(z) = |D \cap \Gamma_z|.
$$
For $1 \leq i \leq 4$, let $\lambda_i = \lambda_i(z)$ be the $i$th successive minimum of $D$ with respect to $\Gamma_z$; i.e. $\lambda_i$ is the smallest $\lambda$ such that $\lambda D$ contains $i$ linearly independent vectors from $\Gamma_z$. By Minkowski's Theorem (Theorem \ref{mink2nd}) we have
\begin{equation}\label{eq:mink2nd}
\lambda_1\lambda_2\lambda_3\lambda_{4} \asymp \frac{p}{H_1^{2}H_2^{2}}
\end{equation}
and by Lemma \ref{henkcount} we have
\begin{equation}\label{eq:rz-bound}
r(z) = |D \cap \Gamma_z| \ll \prod_{j=1}^{4} \max(1, \lambda_j^{-1}).
\end{equation}
For $1 \leq s \leq 4$, let $Z_s$ be the set of $z$ such that $\lambda_s(z) \leq 1$ and $\lambda_{s+1}(z) > 1$ (setting $\lambda_5(z) = +\infty$). If $z$ does not lie in any $Z_s$, then $\lambda_1(z) > 1$ which implies that $D \cap \Gamma_z = \{0\}$ and $r(z)=1$. Thus
$$
\sum_{z \notin Z_1\cup\cdots\cup Z_4} r(z)^2 \leq \sum_{z \in \F_p}r(z) = |A|^2.
$$
Hence it suffices to prove that 
\begin{equation}\label{x-energy-Z_s}
\sum_{z \in Z_s} r(z)^2 \ll \Big(|A|^{2} + \frac{|A|^4}{p}\Big)\log p
\end{equation} 
for each $s \in\{1,2,3,4\}$.

\subsection{Case $s \leq 2$.} \label{sec:proof-case1}

For $z \in Z_s$ with $s \leq 2$, we must have $\lambda_1 = \lambda_1(z) \leq 1$. Pick $v_z = (x_1,x_2,y_1,y_2)$ to be a nonzero vector in $\lambda_1D \cap \Gamma_z$.  Since $v_z= (x_1,x_2,y_1,y_2) \in \lambda_1D$, we have $|x_1|,|y_1| \leq \lambda_1H_1$ and $|x_2|,|y_2| \leq \lambda_1H_2$. Since at least one of the coordinates $x_1,x_2,y_1,y_2$ is nonzero and $H_1 \leq H_2$, we must have $\lambda_1 \geq 1/H_2$. For $\lambda \in [1/H_2, 1]$, define
$$
Z_s(\lambda) = \{z \in Z_s: \lambda \leq \lambda_1(z) \leq \min(2\lambda,1)\}.
$$
If $\lambda \in [1/H_2,1]$ and $z \in Z_s(\lambda)$, then $\lambda_1(z) \leq 2\lambda$ and thus the number of possibilities for $v_z = (x_1,x_2,y_1,y_2)$ is 
$$
\ll (1+\lambda H_1)^2 (1+\lambda H_2)^2 \ll (1+\lambda^2H_1^2) \lambda^2H_2^2.
$$
Note that $z$ is determined uniquely by $v_z$, since otherwise we must have $a_1x_1+a_2x_2 \equiv a_1y_1 + a_2y_2 \equiv 0\pmod{p}$ which implies that $x_1 \equiv x_2 \equiv y_1\equiv y_2\equiv 0\pmod{p}$ by the properness of $A$. Hence it follows that 
\begin{equation}\label{eq:Zs-bound}
|Z_s(\lambda)| \ll (1+\lambda^2H_1^2) \lambda^2H_2^2
\end{equation}
for all $\lambda \in [1/H_2,1]$. From \eqref{eq:rz-bound} we have
$$
r(z) \ll \prod_{j=1}^s \lambda_j^{-1} \ll \lambda_1^{-s} \ll \lambda^{-s}
$$
for $z \in Z_s(\lambda)$. Hence
$$
\sum_{z \in Z_s(\lambda)} r(z)^2 \ll \lambda^{-2s} |Z_s(\lambda)| \ll \lambda^{2-2s}H_2^2 + \lambda^{4-2s} |A|^2.
$$
In the case $s=1$, the bound above is clearly $\ll |A|^2$ and the desired upper bound \eqref{x-energy-Z_s} follows from dyadic summation. In the case $s=2$, the bound above is also $\ll |A|^2$ provided that $\lambda \gg 1/H_1$. Thus for $s=2$ it remains to prove that
\begin{equation}\label{eq:x-energy-Z2}
\sum_{\substack{z \in Z_2 \\ \lambda_1(z) \in [1/H_2,1/H_1]}} r(z)^2 \ll \Big(|A|^{2} + \frac{|A|^4}{p}\Big)\log p.
\end{equation}

For those $z$ included in the summation in \eqref{eq:x-energy-Z2}, we will show that
\begin{equation}\label{eq:lambda2-bound}
\lambda_2(z)^2 \gg \min\Big(\frac{1}{H_1^2}, \frac{p}{H_2^2}\Big).
\end{equation}
Once \eqref{eq:lambda2-bound} is established, we then have from \eqref{eq:rz-bound} that
$$
r(z)^2 \ll \lambda_1(z)^{-2}\lambda_2(z)^{-2} \ll \lambda_1(z)^{-2} \max\Big(H_1^2, \frac{H_2^2}{p}\Big).
$$
Moreover, from \eqref{eq:Zs-bound} it follows that the number of summands in \eqref{eq:x-energy-Z2} with $\lambda_1(z)\sim \lambda$  for some $\lambda \in [1/H_2, 1/H_1]$ is $\ll (1+\lambda^2H_1^2)\lambda^2H_2^2 \ll \lambda^2H_2^2$. Hence for $\lambda \in [1/H_2,1/H_1]$ we have
$$
\sum_{\substack{z \in Z_2 \\ \lambda_1(z) \sim \lambda}} r(z)^2 \ll H_2^2 \max\Big(H_1^2, \frac{H_2^2}{p}\Big) \ll |A|^2 + \frac{|A|^4}{p}.
$$
The desired upper bound \eqref{eq:x-energy-Z2} then follows from dyadic summation.

It remains to prove \eqref{eq:lambda2-bound}. For $z \in Z_2$, pick two linearly independent vectors $v_z = (x_1,x_2,y_1,y_2)$ and $v_z' = (x_1',x_2',y_1',y_2')$ in $\lambda_2D \cap \Gamma_z$. If $\lambda_2(z) \geq 1/H_1$ then \eqref{eq:lambda2-bound} follows immediately. If $\lambda_2(z) < 1/H_1$, then we must have $x_1=y_1=x_1'=y_1'=0$, and thus
$$
zx_2 \equiv y_2\pmod{p}, \ \ zx_2' \equiv y_2'\pmod{p}.
$$
It follows that $x_2y_2'-x_2'y_2 \equiv 0\pmod{p}$. Since $v_z,v_z'$ are two linearly independent vectors, we must have $x_2y_2'-x_2'y_2 \neq 0$ and hence $|x_2y_2'-x_2'y_2| \geq p$. On the other hand, since $|x_2|,|y_2|,|x_2'|,|y_2'| \leq \lambda_2H_2$, we have
$$
p \leq |x_2y_2'-x_2'y_2| \ll (\lambda_2H_2)^2,
$$
which implies that $\lambda_2^2 \gg p/H_2^2$, thus establishing \eqref{eq:lambda2-bound}.

\subsection{Case $s \geq 3$.} \label{sec:proof-case2}

Now we treat the case where $s \in \{3,4\}$. If $s=4$, then from \eqref{eq:mink2nd} and \eqref{eq:rz-bound} we have
$$ r(z) \ll \prod_{j=1}^4 \lambda_j^{-1} \ll p^{-1}H_1^2H_2^2 = p^{-1}|A|^2 $$
for $z \in Z_4$, and thus
$$
\sum_{z \in Z_4} r(z)^2 \ll p (p^{-1}|A|^2)^2 = \frac{|A|^4}{p},
$$
establishing \eqref{x-energy-Z_s}.
Hence it remains to deal with the case $s=3$. For $z \in Z_3$, from \eqref{eq:mink2nd} and \eqref{eq:rz-bound} we have
$$
r(z) \ll \prod_{j=1}^3 \lambda_j^{-1} \ll p^{-1}|A|^2 \lambda_4.
$$
To make effective use of this, we need an upper bound for $\lambda_4$. Let $\Gamma_z^*$ and $D^*$ be the dual lattice and the dual body of $\Gamma_z$ and $D$, respectively. Then
$$
D^* = \{(u_1,u_2,v_1,v_2) \in \R^{4}: H_1|u_1| + H_2|u_2| + H_1|v_1| + H_2|v_2| \leq 1\}.
$$
We claim that
$$
\Gamma_z^* = \Z^{4} + \{ p^{-1}(za_1t, za_2t, -a_1t, -a_2t) : t \in \Z\}.
$$
Clearly the right-hand side above is contained in $\Gamma_z^*$. To establish the other direction, pick any vector $(u_1,u_2,v_1,v_2) \in \Gamma_z^*$. Since $(p,0,0,0),(0,p,0,0), (0,0,p,0),(0,0,0,p) \in \Gamma_z$, we must have $pu_1,pu_2, pv_1,pv_2 \in \Z$.
Since $(0, 0, -a_2,a_1),(1, 0, z, 0), (0, 1, 0, z)\in \Gamma_z$, we must have
$$
a_1v_2 - a_2v_1\in\Z, \ \ u_1+zv_1\in\Z, \ \ u_2+zv_2\in\Z.
$$
Choose $t \in \Z$ such that $-a_1t \equiv pv_1\pmod{p}$. Then the relations above imply that
$$
(pu_1, pu_2, pv_1, pv_2) \equiv (za_1t, za_2t, -a_1t, -a_2t)\pmod{p}.
$$
This proves the claimed description of $\Gamma_z^*$.

Now let $\lambda_1^* = \lambda_1^*(z)$ be the first successive minimum of $\Gamma_z^*$ with respect to $D^*$. By Lemma \ref{polar-minima} we have $\lambda_1^*\lambda_{4} \asymp 1$. Since $\lambda_4 > 1$, we have $\lambda_1^*\ll 1$ and
\begin{equation}\label{eq:rz-bound2}
r(z) \ll p^{-1} |A|^2 (\lambda_1^*)^{-1}.
\end{equation}
We may assume that $\lambda_1^* < 1$, since otherwise we have $r(z) \ll p^{-1}|A|^2$ and the desired estimate \eqref{x-energy-Z_s} follows immediately as in the case $s=4$.

Pick $v_z = (u_1,u_2,v_1,v_2)$ to be a nonzero vector in $\lambda_1^*D^* \cap \Gamma_z^*$. Then $|u_1|,|v_1|\leq \lambda_1^*/H_1$, $|u_2|,|v_2| \leq \lambda_1^*/H_2$, and
$$
(pu_1,pu_2,pv_1,pv_2) \equiv (za_1t, za_2t, -a_1t,-a_2t)\pmod{p}
$$
for some $t \in \Z$. We must have $t \neq 0$, since otherwise $u_1,u_2,v_1,v_2 \in \Z$ and hence they must all be zero, a contradiction. Moreover, we must have $\lambda_1^* \geq 1/p$, since otherwise $|u_1|,|u_2|,|v_1|,|v_2| < 1/p$ and they must all be zero, again a contradiction.

For $1/p \leq \lambda < 1$, define
$$
Z_3(\lambda) = \{z \in Z_3: \lambda/2 \leq \lambda_1^*(z) < \lambda\}.
$$
If $z \in Z_3(\lambda)$ then both $t\pmod{p}$ and $zt\pmod{p}$ lie in the Bohr set
$$
B = B_{\lambda} := \{x \in \F_p: \|a_1x/p\| \leq \lambda/H_1 \text{ and } \|a_2x/p\| \leq \lambda/H_2\},
$$
and thus $z \in B/B$. It follows that $|Z_3(\lambda)| \leq |B_{\lambda}|^2$ and thus
$$
\sum_{z \in Z_3(\lambda)} r(z)^2 \ll |B_{\lambda}|^2 p^{-2} |A|^4 \lambda^{-2}.
$$
We apply Proposition \ref{prop:bohr} with $\Gamma = (a_1,a_2)$ and $\eta = (\lambda/H_1, \lambda/H_2)$ to estimate $|B_{\lambda}|$. Since $B$ contains nonzero elements, we have $t(\Gamma,\eta) \geq 1$. By the properness of $A$, we have $\delta = \delta(\Gamma,\eta) \geq \lambda$. It follows that
$$
|B_{\lambda}| \ll \lambda^{-1} \cdot \frac{\lambda^2}{H_1H_2}p = \frac{\lambda}{|A|}p,
$$
and hence
$$
\sum_{z \in Z_3(\lambda)} r(z)^2 \ll |A|^2,
$$
and the desired estimate \eqref{x-energy-Z_s} follows by dyadic summation. 

\subsection{On multiplicative energies of GAPs of higher rank}\label{sec:general-d}

In this subsection, we briefly point out the obstacle which prevents us from generalizing Theorem \ref{x-energy} to GAPs of rank $d > 2$. Let $A$ be a proper GAP of rank $d > 2$ of the form
$$
A = \{a_1x_1 + \cdots + a_dx_d : |x_i| \leq H\},
$$
where $a_1,\cdots,a_d \in \F_p\setminus\{0\}$ and $H$ is a positive integer. One can still define the lattice $\Gamma_z \subset \Z^{2d}$ for $z \in \F_p$ and the box $D \subset \R^{2d}$ as before, and try to analyze the successive minima $\lambda_1,\cdots,\lambda_{2d}$ of $D$ with respect to $\Gamma_z$. As before, for $1 \leq s \leq 2d$, let $Z_s$ be the set of $z$ such that $\lambda_s(z) \leq 1$ and $\lambda_{s+1}(z) > 1$.

We expect the arguments in Section \ref{sec:proof-case1} for the case $s \leq d$ to go through without difficulties. For the case $s > d$, the arguments in Section \ref{sec:proof-case2} led us to bounding Bohr sets of the form
$$
B = B_{\lambda}:= \{x \in \F_p: \|a_ix/p\| \leq \lambda/H\text{ for each }1 \leq i \leq d\},
$$
where $1/p \leq \lambda < 1$. Defining $Z_s(\lambda)$ as before, we have
$$
\sum_{z \in Z_s(\lambda)} r(z)^2 \ll |B_{\lambda}|^2 \max_{z \in Z_s(\lambda)} r(z)^2.
$$
Analogous to \eqref{eq:rz-bound2} we have
$$
r(z) \ll p^{-1}|A|^2 (\lambda_1^* \cdots \lambda_{2d-s}^*)^{-1} \ll p^{-1}|A|^2 \lambda^{s-2d}
$$
for $z \in Z_s(\lambda)$. Applying Proposition \ref{prop:bohr} with $\Gamma = (a_1,\cdots,a_d)$ and $\eta = (\lambda/H, \cdots, \lambda/H)$ to estimate $|B_{\lambda}|$, we obtain
$$
|B_{\lambda}| \ll \delta^{t-d} \Big(\frac{\lambda}{H}\Big)^d p \ll \frac{\lambda^t}{|A|}p
$$
since $\delta \geq\lambda$. Combining the inequalities above, we get
$$
\sum_{z \in Z_s(\lambda)} r(z)^2 \ll \lambda^{2(s+t-2d)} |A|^2.
$$
This is acceptable if $t = t(\Gamma,\eta)$ satisfies $t \geq 2d-s$, but we are unable to make this connection between $s$ and $t$ for general $d$.

\section{Application to Character Sums}\label{sec:char-sum}

In this section we prove Theorem \ref{char-sum}, the Burgess-type estimate for character sums over rank-$2$ GAPs. Let $A \subset \F_p$ be a proper GAP of rank $2$ of the form
$$
A = \{a_0 + a_1x_1 + a_2x_2: 1 \leq x_i \leq H_i\},
$$ 
where $H_1,H_2$ are positive integers, $a_0 \in \F_p$, and $a_1,a_2 \in \F_p\setminus\{0\}$.  Assume that $|A| \geq p^{1/4+10\eps}$ for some sufficiently small $\eps  > 0$. We may assume that $p$ is sufficiently large in terms of $\eps$, since otherwise the claimed bound is trivial. By writing $A$ as a disjoint union of smaller GAPs, we may assume that $|A| \leq p^{1/2}$ (say). Define
$$
B = \{a_1x_1 + a_2x_2: 1 \leq x_i \leq H_i p^{-2\eps}\} \text{ and }J = [1,p^{\eps}].
$$
We may assume that $H_i \geq p^{5\eps}$ for each $i$, since otherwise $A$ is the disjoint union of arithmetic progressions, each of which has length at least $|A|p^{-5\eps} \geq p^{1/4+5\eps}$, and the desired conclusion follows from Burgess' estimate.

Now for any $y \in B$ and $t \in J$ we have
$$
\Big|\sum_{x \in A}\chi(x) - \sum_{x \in A}\chi(x + yt)\Big| \leq |A \setminus(A +yt)| + |(A+yt) \setminus A| \ll |A|p^{-\eps}.
$$
 Hence 
\begin{equation}\label{char-sum-averages}
\sum_{x \in A}\chi(x) = \frac{1}{|J||B|}\sum_{\substack{y \in B \\ t \in J}}\sum_{x \in A}\chi(x + yt) + O(|A|p^{-\eps}).
\end{equation}
Observe that
$$
\Big|\sum_{\substack{y \in B \\ t \in J}}\sum_{x \in A}\chi(x + yt)\Big| \leq \sum_{\substack{x \in A \\ y \in B}}\Big|\sum_{t \in J}\chi(xy^{-1} + t)\Big| = \sum_{u \in \F_p}\nu(u)\Big|\sum_{t \in J}\chi(u + t) \Big|,
$$
where $\nu(u) = |\{(x,y) \in A\times B : xy^{-1} = u \pmod{p}\}|$. Applying Lemma \ref{char-sum-energy} we obtain
\begin{equation}\label{final-char-estimate}
\frac{1}{|J||B|}\Big|\sum_{\substack{y \in B \\ t \in J}}\sum_{x \in A}\chi(x + yt)\Big| \ll_r |A|^{1-\frac{1}{r}}|B|^{-\frac{1}{r}}(E_{\times}(A)E_\times (B))^{\frac{1}{4r}}(\sqrt{p} + |J|^{-r}p )^{\frac{1}{2r}}
\end{equation}
for any positive integer $r$. Choose $r = \lfloor 1/(2\eps)\rfloor$ so that $|J|^{-r}p \ll \sqrt{p}$.
Since $|A| \leq p^{1/2}$ and $|B| \gg |A|p^{-4\eps}$, by Theorem \ref{x-energy} we have
$$
E_\times(A) \ll |A|^2 \log p, \ \ E_\times (B) \ll |B|^2 \log p.
$$
Inserting these estimates into \eqref{final-char-estimate}, we obtain
$$
 \frac{1}{|J||B|}\Big|\sum_{\substack{y \in B \\ t \in J}}\sum_{x \in A}\chi(x + yt)\Big| \ll_r |A|^{1-\frac{1}{2r}} |B|^{-\frac{1}{2r}} p^{\frac{1}{4r}} (\log p)^{\frac{1}{2r}} \ll |A|^{1-\frac{1}{r}} p^{\frac{2\eps}{r}+\frac{1}{4r}} (\log p)^{\frac{1}{2r}}.
$$
Since $|A| \geq p^{1/4+10\eps}$, the upper bound above is
$$
\ll |A|(\log p)^{\frac{1}{2r}} \cdot p^{-\frac{1}{r}(\frac{1}{4}+10\eps) + \frac{2\eps}{r} + \frac{1}{4r}} \ll |A| p^{-\frac{5\eps}{r}}.
$$
The desired estimate follows by combining this with \eqref{char-sum-averages}.

\bibliographystyle{plain}
\bibliography{biblio}

\end{document}